\documentclass[12pt, reqno]{amsart}
\usepackage{amsmath, amsthm, amscd, amsfonts, amssymb, graphicx, color}
\usepackage{subcaption}
\usepackage[bookmarksnumbered, colorlinks, plainpages]{hyperref}
\hypersetup{colorlinks=true,linkcolor=red, anchorcolor=green, citecolor=cyan, urlcolor=red, filecolor=magenta, pdftoolbar=true}

\newtheorem{theorem}{Theorem}[section]
\newtheorem{lemma}[theorem]{Lemma}
\newtheorem{proposition}[theorem]{Proposition}
\newtheorem{corollary}[theorem]{Corollary}
\theoremstyle{definition}
\newtheorem{definition}[theorem]{Definition}

\theoremstyle{remark}
\newtheorem{remark}[theorem]{Remark}

\numberwithin{equation}{section}

\newtheorem*{theorem-non}{Theorem}

\usepackage{enumerate}
\begin{document}
\setcounter{page}{1}

\title[Paley-Wiener theorem ]{Real Paley-Wiener theorems for the linear canonical Dunkl transform}
 
\author[Umamaheswari S, Sandeep Kumar Verma and Hatem Mejjaoli]{Umamaheswari S\textsuperscript{1}, Sandeep Kumar Verma\textsuperscript{1} and Hatem Mejjaoli\textsuperscript{2}}

 \address{\textsuperscript{1}Department of Mathematics, SRM University-AP, Andhra Pradesh, Amaravati--522240, India\newline
 \textsuperscript{2}Department of Mathematics, Taibah University, College of Science, Al
Madinah Al Munawwarah, 30002, Saudi Arabia
 }
 \email{umasmaheswari98@gmail.com, sandeep16.iitism@gmail.com, \newline hmejjaoli@gmail.com}



\begin{abstract} 
We examine the Sobolev space associated with the linear canonical Dunkl transform and explore some properties of the linear canonical Dunkl operators. Building on these results, we establish a real Paley-Wiener theorem for the linear canonical Dunkl transform. Further, we characterize the square-integrable function $f$ whose linear canonical Dunkl transform of the function is supported in the polynomial domain.   Finally, we develop the Boas-type Paley-Wiener theorem for the linear canonical Dunkl transform.
\end{abstract}
\maketitle
\section{Introduction}
 Paley and Wiener characterized $L^2$
functions on a closed and bounded interval in terms of the image of the Fourier transform of square-integrable functions with compact support. Specifically, their result asserts that a function 
$f\in L^2(\mathbb{R})$ has compact support (i.e., its spectrum is compact) if and only if its Fourier transform  $\hat{f} \in L^2(\mathbb{R})$ extends into $\mathbb{C}$ as an entire function of exponential growth \cite{Paley}. This result is known as the classical complex Paley-Wiener theorem. However, the proof of this theorem does not readily extend to other integral transforms. 
Therefore, an alternative approach was developed by Bang \cite{Bang} to describe the Fourier transform of functions with compact support, avoiding the need for complexification. This approach is known as the real Paley-Wiener theorem, which states that a function $f$ has compact support if and only if
\begin{equation*}
\lim_{n \rightarrow \infty}  \left\|\frac{d^nf}{dx^n}\right\|^{\frac{1}{n}}_{L^p(\mathbb{R})} =  \mathop{\underset{x \in \mathbb{R}}{\text{sup}}}\,\{ |x|: x \in \text{supp}\,\hat{f} \}.
\end{equation*}
 In other words, Bang's theorem characterizes a function 
$f\in L^2(\mathbb{R})$ as having a compact spectrum, with the proof relying on specific properties of the Fourier transform. The term ``real" signifies that the support of the Fourier transform (FT) is determined based on the growth rates of $f$ on $\mathbb{R}$, rather than on $\mathbb{C}$ as in the classical complex Paley-Wiener theorem. Although Bang's theorem characterizes the square-integrable functions in the real setting, his proof does not naturally extend to other integral transforms. Independently, Tuan provided an alternative proof of Bang's theorem, primarily relying on the Parseval formula for the Fourier transform \cite{Tuan}. Building on Tuan’s proof techniques, many authors have extended the real Paley-Wiener theorem to other integral transforms by replacing the operator $\frac{d}{dx}$ with a second-order differential operator, whose eigenfunction acts as a kernel for the integral transformation. 
Simultaneously, Boas characterized the class of all square-integrable functions that vanish on a closed interval in 
$\mathbb{R}$  based on the behavior of their Fourier transform \cite{Boas}. Likewise, Tuan extended Boas's theorem to the class of integral transforms \cite{Zayed} and further explored Paley-Wiener-type theorems for non-convex and unbounded domains \cite{V.K. Tuan}. 
\\
During that period, significant progress was made in studying the real Paley-Wiener theorem in different settings, including the Siegel upper half-space, Clifford analysis,  and generalized functions, as demonstrated in works \cite{Arcozzi, Karunakaran, Kou, Lax}. Additionally, numerous studies have successfully extended the theory to various integral transforms, including the Hankel transform, Dunkl transform, singular Sturm-Liouville integral transform, and others \cite{de Jeu, Fe}.
\par
Notably, Andersen provided an elementary proof of Bang’s theorem \cite{N.B} and characterized the image of compactly supported smooth even functions under the Hankel transform as a subspace of the Schwartz space \cite{Arcozzi}. Later, Abreu extended the real Paley-Wiener theorem to the Koornwinder-Swarttouw $q$-Hankel transform within the framework of quantum calculus \cite{Abreu}. Sequentially, Trim\'eche established the real Paley-Wiener theorem for the Dunkl transform, and in collaboration with Chettaoui, introduced a new type Paley-Wiener theorem for the Dunkl transform \cite{Chettaoui, Trime'Che}. More recently, Li investigated the real-type Paley-Wiener theorem and the Boas-type theorem for the $(k, a)$ -generalized Fourier transform \cite{Li}. 
 These works motivate us to establish the real  Paley-Wiener theorem in the context of the linear canonical Dunkl transform (LCDT).
\par Ghazouani et al. introduced the linear canonical Dunkl transform \cite{ghazouani2017unified} by combining the linear canonical transform and the Dunkl transform. It is the 
generalization of some other integral transforms like the fractional Dunkl
transform \cite{ghazouani2014fractional}, the linear canonical Fourier Bessel transform \cite{ghazouani2023canonical}, the Hankel transform  \cite{TEC}, and the fractional Hankel transform \cite{kerr1991fractional}, etc.
The LCDT is a powerful tool for handling a broader range of mathematical operations, like scaling, rotation, shearing, and allowing for a more versatile mathematical framework. Especially, it is useful in signal analysis, image processing, etc. Signals or functions often have varying degrees of smoothness (i.e., non-smooth signals or functions), so it is necessary to require function spaces with controlled smoothness and decay properties. 
Sobolev spaces possess these properties and serve as an ideal mathematical setting for studying LCDT. Moreover, generalizes the concept of differentiability and integrability for functions, enabling the rigorous analysis of problems where classical derivatives may not exist in the usual sense. In this work, we study linear canonical Dunkl Sobolev spaces and establish fundamental results that are essential for proving the real Paley-Wiener theorem.  The linear canonical Dunkl Sobolev space is defined as for $s \in \mathbb{R}$
\begin{equation*}
{\textbf{W}^{s}_{k,M}(\mathbb{R})} = \{ h \in \mathcal{S' (\mathbb{R}}) : (1+|\lambda|^2)^{\frac{s}{2}}\,\mathcal{D}^M_k(h) \in L^2_k(\mathbb{R})\},
\end{equation*}
where $\mathcal{D}^M_k$ denotes the linear canonical Dunkl transform defined in Definition \ref{D3.1}.
We define a seminorm associated with the linear canonical Dunkl operator 
$$ S_{r,p}(\phi) = C\,\|  (1+x^2)^{\frac{r}{2}}\,\Lambda_{k, M^{-1}}^{p} \varphi \|_{L_k^2(\mathbb{R})}, \quad r,p \in \mathbb{N}^*,$$
where $\phi \in \mathcal{S}(\mathbb{R})$, $M^{-1}$ is the inverse of the matrix $M$, and $\mathbb{N}^*=\mathbb{N}\cup \{0\}$. $\Lambda_{k, M^{-1}}$ denotes the linear canonical Dunkl operator defined in Section \ref{s3} by \eqref{eq:3.1}.
Let us describe our main results:
\begin{theorem-non}[Real Paley-Wiener theorem (Theorem \ref{t:4.1})]
For all $f \in \mathcal{ S(\mathbb{R})} $, the following limit exists
\begin{equation*}
    \lim_{n \rightarrow \infty} \| \Lambda_{k, M^{-1}}^nf\|^{\frac{1}{n}}_{L_k^p(\mathbb{R})} = \sigma_f, \quad 1\le p \le \infty,
\end{equation*}
where $$\sigma_f = \text{sup} \,\left\{ \frac{|\lambda|}{b}:  \lambda \in\text{supp}\,(\mathcal{D}_k^M(f)
)\right\}, \quad b\neq 0.$$
\end{theorem-non}
\begin{theorem-non} [Real Paley-Wiener theorem for non-convex domain (Theorem \ref{t:5.1})]
 The linear canonical Dunkl transform $\mathcal{D}_k^M(f)$ of $f\in \mathcal{S}(\mathbb{R})$   vanishes outside of the polynomial domain $\Omega_{P,b}$ if and only if  
 \begin{equation*} 
 \limsup\limits_{n \to \infty} \|P^n(i\Lambda_{k, M^{-1}})f\|^{\frac{1}{n}}_{L_k^p(\mathbb{R})} \le 1, \quad \text{
 for}\quad 1\le p \le \infty.
 \end{equation*}
\end{theorem-non}
\begin{theorem-non}[Boas type theorem (Theorem \ref{t:5.6})]
Let $f\in \mathcal{S}(\mathbb{R})$ vanishes in some interval $(-r,r)$ if and only if 
\begin{equation*}
     \lim_{n \to \infty}  \Bigg \| \sum_{m=0}^\infty \frac{n^m\, \Delta^m_{k,M^{-1}}f}{m!} 
 \Bigg \|_{L_k^p(\mathbb{R})}^{\frac{1}{n}} \le e^{-r^2}, \qquad \text{for}\, \,\,\,1\le p\le \infty.
\end{equation*}
\end{theorem-non}
The structure of the paper is as follows: Section \ref{S2} introduces the notations and recalls the definitions of the Dunkl transform and Dunkl operator, along with their fundamental results. Section \ref{s3} explores the properties of the linear canonical Dunkl transform and the linear canonical Dunkl operator. Subsequently, in Section \ref{S4}, we define the Sobolev space associated with the linear canonical Dunkl transform and examine some of its properties. In Section \ref{S3}, we establish the real Paley-Wiener theorem for LCDT.  Section \ref{S5} presents the real Paley-Wiener theorem for the Schwartz functions whose linear canonical Dunkl transform is supported within the polynomial domain, and then we establish the Boas type Paley-Wiener theorem for the linear canonical Dunkl transform.

\section{Preliminaries} \label{S2}
In this section, we begin by recalling the definitions of the Dunkl transform and the Dunkl operator. Additionally, we discuss some of their fundamental properties.
\\
$\text{\textbf{Notations:}}$
\begin{itemize}
\item $\mathcal{C}^n (\mathbb{R})$  the space of all $n$ times continuously differentiable  functions  on $\mathbb{R}.$
\vspace{.2cm}
\item $\mathcal{S}(\mathbb{R})$  the space of all infinitely differentiable functions on $\mathbb{R}$ which are rapidly decreasing with their derivatives.
\vspace{.2cm}
\item $\mathcal{C}_0 (\mathbb{R})$ the space of continuous functions on $\mathbb{R}$ which vanishes at infinity.
\vspace{.2cm}
 \item $\mathcal{S'(\mathbb{R})}$ the space of all tempered distributions  on $\mathbb{R}.$
 \vspace{.2cm}
 \item For $1\le p<\infty$, the Lebesgue space $L_k^p(\mathbb{R})$ is  the collection of all measurable functions $f$ on $\mathbb{R}$ such that
\begin{equation*}
\|f\|_{L_k^p(\mathbb{R})} = \left( \int_{\mathbb{R}} |f(x)|^p\, d\mu_k(x)\right)^{\frac{1}{p}}< \infty,
\end{equation*}
where 
\begin{equation*}
    d\mu_k(x) = \frac{
 |x|^{2k+1}\, dx}{2^{k+1}\,\Gamma(k+1)}, \qquad  k\ge -\frac{1}{2}.
\end{equation*}
\vspace{.2cm}
\item 
For $p=\infty$, 
\begin{equation*}
\|f\|_{L^\infty_k(\mathbb{R})} = \mathop{\underset{x \in \mathbb{R}}{\text{ess. sup}}} |f(x)| < \infty.
\end{equation*}
\vspace{.2cm}
\item  We denote by $SL(2,\mathbb{R})$ the set of real unimodular matrices of order $2\times 2$. For notational convenience, we shall write a $2\times 2$ matrix as $M:=(a,b;c,d)$. 
\item  We write $X \lesssim Y$ to indicate that $X \le C Y$, where $C$  is a positive constant that varies depending on the context of the inequality.
\end{itemize}
\subsection{Dunkl operator}
In this section, we recall the Dunkl operator and its fundamental properties, providing a foundation for subsequent analysis of the linear canonical Dunkl operator. \\
C.F. Dunkl introduced the differential-difference operator (Dunkl operator), by associating reflection group and multiplicity function. We are interested in discussing the Dunkl operator on the real line together with the reflection group $\mathbb{Z}_2$. The Dunkl operator defined for a continuously differentiable function on $\mathbb{R}$ as \cite{dunkl1989differential}:
 \begin{equation*}
\Lambda_{k}f(x) = \frac{d}{dx}f(x)+\frac{2k+1}{x}\left[ \frac{f(x)-f(-x)}{2}\right], \quad\quad f\in \mathcal{C}^1(\mathbb{R}).  
\end{equation*}
The second part of the above equation makes sense to call $\Lambda_k$ as the differential-difference operator, which generalizes the classical differential operator by incorporating reflection symmetry. The study of Dunkl operators has experienced significant growth over the past three decades, driven by their importance in various areas of mathematics and physics. Moreover, they provide a framework for developing an analogue of harmonic analysis theory depending on a set of real parameters. A key application of Dunkl operators lies in the analysis of orthogonality structures for polynomials. They have also been instrumental in the development of new classes of polynomials, such as the Euler-Dunkl polynomials \cite{A.J.} and Bernoulli-Dunkl polynomials \cite{J.I.}, among others. More recently, Alejandro et al. investigated the properties of these polynomials and introduced Appell sequences and Hurwitz zeta functions within the Dunkl framework \cite{A.G.}.  In addition, numerous studies have been conducted in this area, further expanding the theory and applications of Dunkl operators \cite{rosler1999uncertainty, Trime'Che}.
\\
The initial value problem is formulated in terms of the Dunkl operator, as given below;
\begin{equation*}
 \left \{
 \begin{array}{ll}
 \Lambda_{k} f(x) = i\lambda f(x),    & 
 \quad\lambda \in \mathbb{R}\mbox{}  \\
 f(0) = 1. & 
 \end{array}
 \right.
\end{equation*}
The function $E_k(i\lambda,x)$ admits a unique solution to the above problem. Observe that for $k=-1/2$ the operator $\Lambda_k$ reduces to the classical differential operator $\Lambda_{-\frac{1}{2}}=\frac{d}{dx}$, and  $E_k(-i\lambda,x)$ simplifies to the kernel of the Fourier transform,  $e^{-i\lambda x}$.\\\\
Now, with the function $E_k(i\lambda,x)$ at hand, we recall the definition of the Dunkl transform. For $f \in  L^{1}_{k}(\mathbb{R}) $, then the Dunkl transform $\mathcal{D}_k$ is defined as \cite{de1993dunkl}
\begin{equation*}
\mathcal{D}_k(f)(\lambda) =  \int_{\mathbb{R}} f(x)\,E_k(-i\lambda,  x)\,  d \mu_k(x),  
\end{equation*}
where $ k\ge -\frac{1}{2},~~~E_k(i\lambda,x)$ denotes the Dunkl kernel   defined as \cite{rosler1999uncertainty}
 \begin{equation*}
  E_k(i\lambda,x) = j_{k}(\lambda x)+\frac{i\lambda x}{2(k+1)}j_{k+1}(\lambda x),\end{equation*}
 and $j_{k}$  denotes the normalized spherical Bessel function \cite{rosler2000one, watson1922treatise}
 \begin{equation*}
  j_{k}(x) = 2^{k} \Gamma(k+1)x^{-k}J_{k}(x) = \Gamma(k +1)\sum^{\infty}_{n=0}\frac{(-1)^n(x/2)^{2n}}{n!\Gamma(n+k+1)}.
  \end{equation*}
Over time, researchers have developed the theory of the Dunkl transform in multiple directions, such as the wavelet transform \cite{Prasad}, uncertainty principles \cite{Sraieb}, localization operators \cite{Mejjaoli}, etc.
Some essential properties of the Dunkl operator and their relation to the Dunkl transform are listed below \cite{Trime'Che}:
\begin{proposition}
\begin{itemize}
\item [$(i)$]Let $f,g \in \mathcal{ S(\mathbb{R})}$. Then 
  \begin{equation*}
   \int_{\mathbb{R}} \Lambda_k f(x) \,\overline{g(x) } \, d\mu_k(x )= -\int_{\mathbb{R}} f(x)\, \overline{\Lambda_k g(x)}\,  d\mu_k(x).
\end{equation*} \item[$(ii)$]  For all $f \in \mathcal{ S(\mathbb{R})}$ and $n \in \mathbb{N}$, we have
\begin{equation} \label{e3.1}
\mathcal{D}_k(\Lambda^n_k f)(\lambda) = (-i\lambda)^n\,\mathcal{D}_k(f)(\lambda).
\end{equation}
\end{itemize}
\end{proposition}
We recall the following proposition, which is essential for establishing the same result for the linear canonical Dunkl operator.
\begin{proposition} \cite[Proposition 2.1]{Chettaoui}\label{p:3.4} Let $f\in \mathcal{C}^\infty(\mathbb{R})$. Then
 \begin{itemize}
     \item [$(i)$] For all $n \in \mathbb{N}$, the Dunkl operator $ \Lambda_k^n f\in \mathcal{C}^\infty(\mathbb{R})$.\\
     \item[$(ii)$] For every $m\in \mathbb{N}^*$ and $R>0$, there exist a constant $C_m$ such that for all $x\in \left[ -R,R\right]$ and some $\xi_j(x,m),$\quad $j=0,1, \cdots m$,  satisfying 
 \begin{equation*}
 | \Lambda_k^mf(x)| \le |f^{(m)}(x)| +C_m\, \sum_{j=1}^{m} |f^{(m)}(\xi_j)|.    
\end{equation*}
\item[$(iii)$] For all $m\in \mathbb{N}^*$, and $R>0$ there exist  $C_m^0>0$ such that
\begin{equation*}
| \Lambda_k^mf(x)| \le |f^{(m)}(x)| +C_m^0\, \sum_{j=0}^{m-1} |f^{(j)}(x)|+|f^{(j)}(-x)|, \,\, \forall x \in \mathbb{R}\quad \text{with}\quad |x|>R.
\end{equation*}
 \end{itemize}   
\end{proposition}
\section{Linear canonical Dunkl operator}\label{s3}
We introduced the linear canonical Dunkl operator by incorporating the Dunkl operator $\Lambda_k$ and matrix parameter $M$, denoted as $\Lambda_{k, M}$. For $f\in \mathcal{C}^1(\mathbb{R})$, then the linear canonical Dunkl operator is defined as \cite{USK}
\begin{equation} \label{eq:3.1}
\Lambda_{k, M}(f)(x) = \Lambda_kf(x)-i\left(\frac{d}{b}\right) x\, f(x).   
\end{equation}
 The linear canonical Dunkl Laplacian is defined in the usual way as $ \Delta_{k, M} = \Lambda^2_{k, M}$. By selecting appropriate matrix parameters, 
$\Lambda_{k, M}$ recovers several well-known operators. Specifically, for
\begin{enumerate}[$(i)$]
    \item  $M = (
 \cos\theta, -\sin\theta;
 \sin\theta, \cos\theta),\, \theta \in \mathbb{R}, \theta\neq n\pi$, 
 $\Lambda_{k, M}$  reduces to the fractional Dunkl operator \cite{ghazouani2014fractional} 
 \item $M= (0,-1;1,0)$ reduces to Dunkl operator \cite{dunkl1989differential}.
\end{enumerate}
Thus, the linear canonical Dunkl operator encompasses various differential operators within a unified structure. Due to the versatile nature of $\Lambda_{k, M}$, it may offer a more comprehensive framework for studying the symmetry structures of various spaces compared to the  Dunkl operator. \\
We set the following eigenvalue problem by exploiting the linear canonical Dunkl operator: For $\lambda \in \mathbb{R}$
\begin{equation*}
 \left \{
 \begin{array}{ll}
 \Lambda_{k, M} f(x) =-i\left(\frac{\lambda}{b}\right) f(x),     & \mbox{}  \\
 f(0) = e^{\frac{i}{2} \frac{d}{b} \lambda^2}, & 
 \end{array}
 \right.
\end{equation*}
has a unique solution $x\mapsto E^M_{k}(\lambda,x).$
Now, we are in the position to recall the definition of the linear canonical Dunkl transform.  
\begin{definition}\label{D3.1}
Let $f \in L_k^1(\mathbb{R})$. We define the linear canonical Dunkl transform (LCDT) as follows \cite{ghazouani2017unified, USK}:
\begin{equation*}
\mathcal{D}_k^M(f)(\lambda) = \frac{1}{(ib)^{k+1}} \int_{\mathbb{R}}f(x)E^M_{k}(\lambda,x)d\mu_{k}(x),  \quad b \neq 0,
\end{equation*}
and the kernel $E^M_{k}(\lambda,x) $ is given by
\begin{equation*}
E^M_{k}(\lambda,x) = e^{\frac{i}{2}(\frac{d}{b}\lambda^2+\frac{a}{b}x^2)}E_{k}(-i\lambda/b,x).
\end{equation*} 
\end{definition}
We listed some basic properties of the linear canonical Dunkl transform, and most of the results were taken from \cite{ghazouani2017unified, USK}.
\begin{proposition}
\begin{enumerate}[$(i)$]
\item 
\textbf{Plancherel's formula:} If $f \in  L^{1}_{k}(\mathbb{R})\cap L^{2}_{k}(\mathbb{R}) $, then  $\mathcal{D}^M_k(f) \in L^{2}_{k}(\mathbb{R})$ and  
 \begin{equation}
 \|\mathcal{D}^M_k(f)\|_{ L^{2}_{k}(\mathbb{R})} = \| f \|_{ L^{2}_{k}(\mathbb{R})}.\label{eq :2.1}
 \end{equation}
 \vspace{.2cm}
\item \textbf{Parsavel's formula:} For $f,g \in L_k^2(\mathbb{R})$, we have
\begin{equation}
 \int_{\mathbb{R}}f\label{eq3.3}(x)\,\overline{g(x)}\,d\mu_k(x) = \int_{\mathbb{R}} \mathcal{D}_k^M(f)(\lambda)\,\overline{\mathcal{D}_k^M(g)(\lambda)}\,d\mu_k(\lambda).   
\end{equation}
\vspace{.2cm}
\item \textbf{Inverse formula:}  For every $ f\in L^1_{k}(\mathbb{R})$ such that $\mathcal{D}^M_{k}f \in L^1_{k}(\mathbb{R})$, then   we have
\begin{equation} \label{e:2.2}
\mathcal{D}^{M^{-1}}_{k} (\mathcal{D}^M_{k}f)
=f ,\,\, \text{a.e}.    
\end{equation}
\vspace{.2cm}
 \item\textbf{Hausdorff-Young's inequality:} For $ 1 \le p \le 2$, the operator $\mathcal{D}_k^M$ is a bounded linear operator on $L_k^p(\mathbb{R})$, and it satisfies the following inequality:   \begin{equation}\label{eq:2.3}
    \|\mathcal{D}_k^M(f)\|_{L_k^{q}(\mathbb{R})} \le \frac{1}{\left(|b|^{k+1}\right)^{1-\frac{2}{q}}}\, \|f\|_{L^p_k(\mathbb{R})},
\end{equation}
where $q$ is the conjugate exponent of $p$.\\
\vspace{.2cm}
\item \textbf{Reimann-Lebesgue lemma:}
    For every $f\in L_k^1(\mathbb{R}),$ we have $\mathcal{D}_k^M(f)\in \mathcal{C}_0(\mathbb{R})$ and
    \begin{align}\label{e2.3}
    \|\mathcal{D}_k^M(f)\|_{L_k^{\infty}(\mathbb{R})} \leq \frac{c_k}{|b|^{k+1}} \|f\|_{L_k^1(\mathbb{R})}.
    \end{align}
\item The map $\mathcal{D}_k^M:\mathcal{S}(\mathbb{R}) \rightarrow \mathcal{S}(\mathbb{R})$ is a topological isomorphism.
\end{enumerate}
\end{proposition}
The linear canonical Dunkl operator  $\Lambda_k^M$  shares many properties with the Dunkl transform.
\begin{proposition}
\begin{itemize}
\item [$(i)$]Let $f,g \in \mathcal{ S(\mathbb{R})}$. Then  we have 
\begin{equation}\label{eq:2.4}
    \int_{\mathbb{R}} \Lambda^M_k f(x)\, \overline{ g(x)}\,  d\mu_k(x) = -\int_{\mathbb{R}} f(x)\,\overline{\Lambda^M_k
    g(x)}\,  d\mu_k(x).
\end{equation}

\item [$(ii)$] For $f\in \mathcal{ S(\mathbb{R})} $, we have 
\begin{equation} \label{e:2.3}
\mathcal{D}_k^M \left( \Lambda_{k, M^{-1}}^nf\right)(\lambda) = \left( \frac{i\lambda}{b}\right)^n \, \mathcal{D}_k^M(f)(\lambda),
\end{equation}
for all $n \in \mathbb{N}^*$.
\item[$(iii)$] For $f \in \mathcal{S}(\mathbb{R})$ and fixed $n \in \mathbb{N}$, we have 
\begin{equation} \label{e:3.6}
    \mathcal{D}_k^M \left( \sum_{m=0}^\infty \frac{n^m\, \Delta^m_{k,M^{-1}}f}{m!} \right) (\lambda) = e^{-n\lambda^2/b^2}\, \mathcal{D}_k^Mf(\lambda).
\end{equation}
\item[$(iv)$] The space $\mathcal{S}(\mathbb{R})$ is invariant under the linear canonical Dunkl operator $\Lambda_{k,M}$.
\end{itemize}
\end{proposition}
\begin{proof}
The proof follows directly from the definitions of the linear canonical Dunkl transform and the linear canonical Dunkl operator.
\end{proof}
\begin{remark} \label{r:3.1}
This remark emphasizes the significant link between the linear canonical Dunkl transform and the Dunkl transform, as well as between the linear canonical Dunkl operator and the Dunkl operator.
\begin{itemize}
 \item[$(i).$] The linear canonical Dunkl transform of 
$f$ can be interpreted in the context of the Dunkl transform of $\tilde{f}$ in the following manner:
\begin{eqnarray*} 
 \mathcal{D}_k^M(f)(\lambda) 
 &=& \frac{e^{\frac{i}{2}\frac{d}{b}\lambda^2}}{(ib)^{k+1}}\, \mathcal{D}_k(\tilde{f})\left(\frac{\lambda}{b}\right),~b\neq 0,
\end{eqnarray*}
where $\tilde{f}(x) =  e^{\frac{i}{2}\frac{a}{b}x^2}f(x).$\\
\item[$(ii).$] For $f\in \mathcal{S}(\mathbb{R})$, we have
$$ \Lambda_{k, M^{-1}}(f)(x) = e^{-\frac{i}{2}\frac{a}{b}x^2}\, \Lambda_k \left( e^{\frac{i}{2}\frac{a}{b}x^2}f\right)(x).$$
\end{itemize}
\end{remark}
We prove the following Proposition by using Proposition \ref{p:3.4}.
\begin{proposition} \label{p:2.1}
Let $f \in C^\infty(\mathbb{R})$. Then
\begin{itemize}
\item [$(i)$] For every $n\in \mathbb{N}^*$, $\Lambda_{k, M^{-1}}^nf$  belongs to $C^\infty(\mathbb{R})$.\\
\item[$(ii)$] Let  $R>0$ and $n\in \mathbb{N}^*$. For every $x\in [-R,R]$, there exist  $\xi_i,$ \newline $i=0,1,2, ..., n$ which depends on $x$ and $n$, such that 
\begin{equation*}
|\Lambda_{k, M^{-1}}^nf(x)| \lesssim   \sum_{j=0}^{n} \binom{n}{j}\,|f^{(n-j)}(x)|+ C_n \sum_{i=0}^n \sum_{j=0}^{n}|f^{(n-j)}(\xi_i)|,
\end{equation*}
for some positive constant  $C_n$.\\
\item[$(iii)$] For $R>0$ and $n\in \mathbb{N}^*$. If $x\in \mathbb{R}$ with $|x|> R$, then there exist $C_n^0>0$ such that 
\begin{eqnarray*}
|\Lambda_{k, M^{-1}}^nf(x)| &\lesssim&\sum_{j=0}^{n} \binom{n}{j}\,|P_j(x)|\,|f^{(n-j)}(x)|+ C^0_n\, \sum_{i=0}^{n-1} \sum_{j=0}^{i} \binom{i}{j}\,|P_j(x)|\,|f^{(i-j)}(x)|\\
 &+&C^0_n\,\sum_{i=0}^{n-1}\sum_{j=0}^{i} \binom{i}{j}\,|P_j(-x)|\,|f^{(i-j)}(-x)|.
\end{eqnarray*}
where  $P_k(x)$ is a polynomial of degree at most $k$.
\end{itemize}
\end{proposition}
\begin{proof}
The proof of $(i)$ immediately follows from the definition of $\Lambda_{k,M^{-1}}$.\\\\
$(ii).$ We observe from Remark \ref{r:3.1}$(ii)$ that
$$ \Lambda_{k, M^{-1}}(f)(x) = e^{-\frac{i}{2}\frac{a}{b}x^2}\, \Lambda_k \left( e^{\frac{i}{2}\frac{a}{b}x^2}f\right)(x).$$
This implies
\begin{equation*}
  \Lambda^n_{k, M^{-1}}(f)(x) = e^{-\frac{i}{2}\frac{a}{b}x^2}\, \Lambda^n_k \left( e^{\frac{i}{2}\frac{a}{b}x^2}f\right)(x), \quad n \in \mathbb{N}^*. 
\end{equation*}
Let us take $g(x) =e^{\frac{i}{2}\frac{a}{b}x^2}f(x)$. A simple calculation gives that 
\begin{equation}\label{e:3.9}
    g^{(n)}(x) = e^{\frac{i}{2}\frac{a}{b}x^2}\, \sum_{j=0}^{n} \binom{n}{j}\,P_j(x)\,f^{(n-j)}(x),
\end{equation}
where $P_j(x)$ is a polynomial of degree $j$. We consider
\begin{eqnarray*}
|  \Lambda^n_{k, M^{-1}}(f)(x)| = |   \Lambda^n_k \left( e^{\frac{i}{2}\frac{a}{b}x^2}f\right)(x)| = |\Lambda^n_k (g)(x)|.
\end{eqnarray*}
Using Proposition \ref{p:3.4}$(ii)$ and \eqref{e:3.9}, we deduce that
\begin{eqnarray*}
|  \Lambda^n_{k, M^{-1}}(f)(x)| &\le& |g^{(n)}(x)|+C_n \sum_{i=0}^n |g^{(n)}(\xi_i)| 
\end{eqnarray*}
\begin{eqnarray*}
&=& \sum_{j=0}^{n} \binom{n}{j}\,|P_j(x)|\,|f^{(n-j)}(x)|+ C_n \sum_{i=0}^n \sum_{j=0}^{n} \binom{n}{j}\,|P_j(\xi_i)||\,f^{(n-j)}(\xi_i)|.
\end{eqnarray*}
Since $x\in \left[R,-R \right]$, we can bound $|P_j(x)|$ by some constant. Thus,
\begin{eqnarray*}
|  \Lambda^n_{k, M^{-1}}(f)(x)| \lesssim  \sum_{j=0}^{n} \binom{n}{j}\,|f^{(n-j)}(x)|+ C_n \sum_{i=0}^n \sum_{j=0}^{n}|f^{(n-j)}(\xi_i)|.
\end{eqnarray*}
\\
$(iii).$ For $|x|>R$ and using Proposition \ref{p:3.4}$(iii)$, we derive that
\begin{eqnarray*}
| \Lambda^n_{k, M^{-1}}(f)(x)| &\le&  |g^{(n)}(x)| +C^0_n\, \sum_{i=0}^{n-1} \left(|g^{(i)}(x)|+|g^{(i)}(-x)|\right)\\
&\lesssim&\sum_{j=0}^{n} \binom{n}{j}\,|P_j(x)|\,|f^{(n-j)}(x)|+ C^0_n\, \sum_{i=0}^{n-1} \sum_{j=0}^{i} \binom{i}{j}\,|P_j(x)|\,|f^{(i-j)}(x)|\\
 &+&C^0_n\,\sum_{i=0}^{n-1}\sum_{j=0}^{i} \binom{i}{j}\,|P_j(-x)|\,|f^{(i-j)}(-x)|.
\end{eqnarray*}
Hence, we obtain the desired result.  
\end{proof}
\section{Linear canonical Dunkl Sobolev space} \label{S4}
In many mathematical and physical problems, classical function spaces such as $L^p$ or 
$\mathcal{C}^k$ often fails to capture essential properties of solutions to differential equations, particularly when dealing with weak derivatives, irregular functions, or variational methods. Sobolev spaces provide a framework to overcome these limitations by incorporating function values and their derivatives in a generalized sense. This subsection introduces the Sobolev space associated with the linear canonical Dunkl operator, which may extend applications to complex systems governed by partial differential equations. By exploring the Sobolev space, we aim to develop a deeper theoretical understanding of function behavior and facilitate further advancements in mathematical analysis and its applications in physics.
\begin{definition}
Let $s \in \mathbb{R},$ we define the Sobolev space associated with the linear canonical Dunkl transform. 
$${\textbf{W}^{s}_{k,M}(\mathbb{R})} = \{ h \in \mathcal{S' (\mathbb{R}}) : (1+|\lambda|^2)^{\frac{s}{2}} \mathcal{D}^M_k(h) \in L^2_k(\mathbb{R})\}. $$
We define the inner product $\langle \cdot,\cdot \rangle_{\textbf{W}^{s}_{k,M}(\mathbb{R})}$  as follows
\begin{equation*}
\langle f,g \rangle_{\mathbf{W}^{s}_k(\mathbb{R})} = \int_{\mathbb{R}}  (1+|\lambda |^2)^s\, \mathcal{D}^M_k(f)(\lambda)\,\overline{\mathcal{D}^M_k(g)(\lambda)}\,d \mu_k(\lambda), \qquad\forall f,g \in \textbf{W}^{s}_{k,M}(\mathbb{R}).   \label{eq: 5.1}  
\end{equation*}
\end{definition}
For $k = -\frac{1}{2}$, the linear canonical Dunkl  Sobolev space reduces to the Sobolev space in the theory of the Fourier transform.

\begin{proposition}\label{P :5.3}
    The space $\mathbf{W}^{s}_{k,M}(\mathbb{R}), \,\,s\in \mathbb{R}$ equipped with the inner product $\langle \cdot,\cdot \rangle_{\mathbf{W}^{s}_{k,M}(\mathbb{R})}$ is a Hilbert space.
\end{proposition}
\begin{proof}
    The proof of the proposition is straightforward and hence it is omitted.
\end{proof}

\begin{remark} We observe that
    \begin{itemize}
        \item [(i)] For all $s,t \in \mathbb{R}$, such that $t>s$, the Sobolev space $\mathbf{W}^{t}_{k,M}(\mathbb{R})$ contained in $\mathbf{W}^{s}_{k,M}(\mathbb{R})$.\\
        \item[(ii)] $\mathbf{W}^{0}_{k,M}(\mathbb{R}) = L^2_k(\mathbb{R}).$
    \end{itemize}
\end{remark}
The forthcoming proposition reveals another norm on $\mathbf{W}^{s}_{k,M}(\mathbb{R})$. 
\begin{proposition}We have the following results.
\begin{itemize}
\item [$(i)$] For $m\in \mathbb{N}^*$, the space $\mathbf{W}^{m}_{k,M}(\mathbb{R})$ is equal to 
\begin{equation*}
E_m=\left\{ f\in L_k^2(\mathbb{R}): \Lambda_{k, M^{-1}}^nf \in L_k^2(\mathbb{R}), \forall n\le m\right\}.     
\end{equation*}
\item [$(ii)$] $ \|f\|^2_{\mathbf{W}^{s}_{k,M}(\mathbb{R})} = \sum_{n=0}^{m}\|\Lambda_{k, M^{-1}}^nf \|^2_{L_k^2(\mathbb{R})}.$
\end{itemize}
\end{proposition}
\begin{proof}
Let $f\in \mathbf{W}^{m}_{k,M}(\mathbb{R})$. For $n \le m$, we have the following inequality
\begin{eqnarray*}
\int_{\mathbb{R}}|\lambda|^{2n}\,|\mathcal{D}_k^M(f)(\lambda)|^2\, d\mu_k(\lambda) \le \int_{\mathbb{R}}(1+|\lambda|^2)^m\,|\mathcal{D}_k^M(f)(\lambda)|^2\,d\mu_k(\lambda).
\end{eqnarray*}
Using \eqref{e:2.3}, we obtain that
\begin{eqnarray*}
|b|^{2n}\int_{\mathbb{R}}|\mathcal{D}_k^M(\Lambda_{k, M^{-1}}^nf)(\lambda)|^2\, d\mu_k(\lambda) &\le& \|f\|^2_{\mathbf{W}^{m}_{k,M}(\mathbb{R})}.
\end{eqnarray*}
Thus
\begin{eqnarray*}\|\Lambda_{k, M^{-1}}^n(f)
\|^2_{L_k^2(\mathbb{R})} &\le& \frac{1}{|b|^{2n}}\,\|f\|^2_{\mathbf{W}^{m}_{k,M}(\mathbb{R})}.
\end{eqnarray*}
So
\begin{eqnarray*} \sum_{n=0}^{m}\|\Lambda_{k, M^{-1}}^n(f)\|^2_{L_k^2(\mathbb{R})}&\le & C(b)\,\|f\|^2_{\mathbf{W}^{m}_{k,M}(\mathbb{R})}.
\end{eqnarray*}
On the other hand, let $f\in E_n$. From Plancherel's formula and simple calculation, we have the following inequalities
\begin{eqnarray*}
\int_{\mathbb{R}}(1+|\lambda|^2)^m\, |\mathcal{D}_k^M(f)(\lambda)|^2\, d\mu_k(\lambda) &\le& \int_{\mathbb{R}}C\left(1+\sum_{n\le m}^{}|\lambda|^{2n}\right)|\mathcal{D}_k^M(f)(\lambda)|^2\,d\mu_k(\lambda)
\end{eqnarray*}
\begin{eqnarray*}
&=& C\|f\|_{L_k^2(\mathbb{R})}+C\sum_{n\le m}^{}\int_{\mathbb{R}} |\lambda|^{2n}\,|\mathcal{D}_k^M(f)(\lambda)|^2\,d\mu_k(\lambda)\\
&=& C\|f\|_{L_k^2(\mathbb{R})}+ C\,|b|^{2n}
\sum_{n\le m}^{}\int_{\mathbb{R}}|\mathcal{D}_k^M(\Lambda_{k, M^{-1}}^nf)(\lambda)|^2\, d\mu_k(\lambda)\\
&=& C_2(b)\,\sum_{n=0}^{m} \int_{\mathbb{R}}|\Lambda_{k, M^{-1}}^n(f)(\lambda)|^2\,d\mu_k(\lambda)\\
\|f\|_{\mathbf{W}^{m}_{k,M}(\mathbb{R})} &\le& C_2(b)\, \sum_{n=0}^{m}\|\Lambda_{k, M^{-1}}^n(f)\|^2_{L_k^2(\mathbb{R})}.
\end{eqnarray*}
This completes the proof.
\end{proof}
\begin{theorem} \label{t:2.6}
Let $p\in \mathbb{N}$    and $  s\in \mathbb{R}$, such that $s>k+p+1$. Then
$\mathbf{W}^{s}_{k,M}(\mathbb{R})$ continuously embedded in $\mathcal{C}^p(\mathbb{R})$.
\end{theorem}
\begin{proof}
The proof is divided into two parts.  First, we  prove the  set inclusion $\mathbf{W}^{s}_{k,M}(\mathbb{R})\subset \mathcal{C}^p(\mathbb{R})$. Then, we will show the continuity from  $\mathbf{W}^{s}_{k,M}(\mathbb{R})$ to $\mathcal{C}^p(\mathbb{R})$. Let $f\in \mathbf{W}^{s}_{k,M}(\mathbb{R})$ and let us consider
\begin{eqnarray*}
\int_{\mathbb{R}} |\mathcal{D}_k^M(f)(\lambda)|\,d\mu_k(\lambda)= \int_{\mathbb{R}} (1+|\lambda|^2)^{\frac{s}{2}}\,(1+|\lambda|^2)^{\frac{-s}{2}}\,|\mathcal{D}_k^M(f)(\lambda)|\,d\mu_k(\lambda).
\end{eqnarray*}
Applying the Cauchy-Schwarz inequality, we obtain 
\begin{eqnarray*}
\|\mathcal{D}_k^M(f)\|_{L_k^1(\mathbb{R})} \le C\,\|f\|_{\mathbf{W}^{s}_{k,M}(\mathbb{R})}.
\end{eqnarray*}
This implies $\mathcal{D}_k^M(f) \in L_k^1(\mathbb{R})$. Doing some manipulation we see that $f\in L_k^2(\mathbb{R})$ and Plancherel's  formula guaranty that $\mathcal{D}_k^M(f) \in L_k^2(\mathbb{R})$. So we have $\mathcal{D}_k^M(f) \in L_k^1 (\mathbb{R}) \cap L_k^2(\mathbb{R}) $ and apply the  Riemann-Lebesgue lemma for the inverse linear canonical Dunkl transform implies that
$f\in \mathcal{C}_0(\mathbb{R})$. Therefore,  $f\in \mathcal{C}(\mathbb{R})$. It remains to show that $f$ is $p$ times continuously differentiable. From \eqref{e:2.2}, we have
\begin{equation*}\label{e:2.7}
f(x)= \frac{1}{(-ib)^{k+1}}\int_{\mathbb{R}} \mathcal{D}_k^M(f)(\lambda)\,E_k^{M^{-1}}(\lambda,x)\,d\mu_k(x),\,\, \text{a.e}.
\end{equation*}
 By using Remark \ref{r:3.1}, we deduce that
 \begin{eqnarray} \label{e3.4}
  f(x) = \frac{1}{|b|^{2k+2}}  \int_{\mathbb{R}}\mathcal{D}_k(f)(\lambda)\, E_k\left(\frac{i\lambda}{b},x\right)\, d\mu_k(\lambda).
 \end{eqnarray}
 Here, we recall the estimate for the derivative of the Dunkl kernel \cite{Trime'Che}
 \begin{equation} \label{e3.5}
\left|\frac{d^n}{dx^n}E_k(i\lambda,x)\right| \le |\lambda|^n.
 \end{equation}
 By differentiating \eqref{e3.4} $n$ times with respect to $x$ and applying  the dominated convergence theorem, we can push the differentiation inside the integration, then we get
\begin{eqnarray} \label{e3.6}
 \frac{d^n}{dx^n}f(x) = \frac{1}{|b|^{2k+2}}\int_{\mathbb{R}} \mathcal{D}_k(f)(\lambda)\, \frac{\partial^n}{\partial x^n}E_k(i\lambda/b,x)\,d\mu_k(\lambda),   \end{eqnarray}
for $n \le p$. We immediately see
 that $f \in \mathcal{C}^p(\mathbb{R})$, for $n\le p$. Now, we proceed to prove the continuity between $\mathcal{C}^p(\mathbb{R})$ and $\mathbf{W}^{s}_{k,M}(\mathbb{R})$. By using \eqref{e3.5}, we can deduce that
 \begin{equation} \label{e3.11}
   \left| \frac{d^n}{dx^n}f(x) \right| \le C(b)\,\|f\|_{\mathbf{W}^{s}_{k,M}(\mathbb{R})}, \quad \text{for}\quad s>k+p+1,\,\, n\le p.
 \end{equation}
This implies that the derivative of $f$ is bounded. We have the norm on $\mathcal{C}^p(\mathbb{R})$ defined as
\begin{eqnarray*}
 \|f\|_{\mathcal{C}^p(\mathbb{R})} = \sum_{n\le p}\, \sup_{x\in \mathbb{R}} \,\left| \frac{d^n}{dx^n}f(x) \right|.
\end{eqnarray*}
Invoking \eqref{e3.11}, we obtain
\begin{equation*}
 \|f\|_{\mathcal{C}^p(\mathbb{R})} \le  C'\,\|f\|_{\mathbf{W}^{s}_{k,M}(\mathbb{R})}.   
\end{equation*}
This completes the proof. 
\end{proof}
Now, we define the new seminorms on $\mathcal{S}(\mathbb{R})$ in terms of the linear canonical Dunkl operator $\Lambda_{k, M^{-1}}$. Before that, we recall some equivalent seminorms on the Schwartz space  $\mathcal{S}(\mathbb{R})$. It will help us to prove the upcoming theorem. Let $n,m,p,q \in \mathbb{N}^*$. Then we have the two sets of equivalent seminorms $S_{n,m}$ and $R_{p,q}$ which are defined below:
\begin{eqnarray*}
S_{n,m}(\phi) &=&\mathop{\underset{ x\in \mathbb{R}}{\sup}}\,(1+x^2)^n\,\left|\frac{d^{m}}{dx^{m}} \phi(x) \right|\\
R_{p,q}(\phi) &=& \mathop{\underset{ x\in \mathbb{R}}{\sup}}\,\left|\frac{d^{p}}{dx^{p}} \left((1+x^2)^q\,\phi(x)\right) \right|.
\end{eqnarray*}

\begin{theorem} \label{t3.8}
 Let $\phi \in \mathcal{S}(\mathbb{R})$ and $n, m \in \mathbb{N}^*$. Define the norm
\begin{equation*}
\|(1+x^2)^m\,\Lambda_{k, M^{-1}}^n (\phi)\|_{L_k^2(\mathbb{R})}= \left( \int_{\mathbb{R}}|(1+x^2)^m\,\Lambda_{k, M^{-1}}^n(\phi)(x)|^2\,d\mu_k(x)\right)^{\frac{1}{2}}.
\end{equation*}
Then, the collection of these norms, as $n, m \in \mathbb{N}^*$, defines a family of seminorms that generates the topology on  $\mathcal{S}(\mathbb{R})$.
\end{theorem}
\begin{proof}
Assume that   $\phi \in \mathcal{S}(\mathbb{R})$. Then, we can find a positive constant $C$ and $r\in \mathbb{N}^*$ such that
\begin{equation*}
\|(1+x^2)^m\,\Lambda_{k, M^{-1}}^n (\phi)\|_{L_k^2(\mathbb{R})} \le C\,\mathop{\underset{ x\in \mathbb{R}}{\sup}} (1+x^2)^r \,|\Lambda_{k, M^{-1}}^n(\phi)(x)|.
\end{equation*}
Invoking Proposition \ref{p:2.1}, there exist  $p \in \mathbb{N}$, thus we deduce that
\begin{eqnarray*}
\|(1+x^2)^m\,\Lambda_{k, M^{-1}}^n (\phi)\|_{L_k^2(\mathbb{R})} \lesssim \mathop{\underset{ x\in \mathbb{R}}{\sup}}\,(1+x^2)^r\, |\phi^p(x)|.
\end{eqnarray*}
This implies 
\begin{equation}\label{e:2.9}
\|(1+x^2)^m\,\Lambda_{k, M^{-1}}^n (\phi)\|_{L_k^2(\mathbb{R})} \lesssim  S_{r,p}(\phi).    
\end{equation}
On the other hand, let us consider
\begin{eqnarray*}
 S_{r,p}(\phi) &=&  \mathop{\underset{ x\in \mathbb{R}}{\sup}} \,(1+x^2)^r\, |\phi^p(x)|,\quad \text{for some}\,\, r,p\in \mathbb{N}^*,\\
&\le&  \mathop{\underset{ x\in \mathbb{R}}{\sup}}\, \left|\frac{d^{p}}{dx^{p}}( (1+x^2)^{q}\,\phi(x))\right|, \quad \text{for some}\,\, p, q\in \mathbb{N}^*.
\end{eqnarray*}
By using Theorem \ref{t:2.6}, there exist a constant $C>0$, $q,l\in \mathbb{N}$ and $l>k+p+1$ such that
\begin{eqnarray*}
\mathop{\underset{ x\in \mathbb{R}}{\sup}}\, \left|\frac{d^{p}}{dx^{p}}\left( (1+x^2)^{q}\,\phi\right)\right| \le C\, \|(1+x^2)^q\,\phi\|_{\mathbf{W}^{s}_k(\mathbb{R})}
\end{eqnarray*}
\begin{eqnarray*}
&=&C \int_{\mathbb{R}}(1+|\lambda|^2)^{-l}\,(1+|\lambda|^2)^{2l}\, \left|\mathcal{D}_k\left( (1+x^2)^q \phi\right) \left(\frac{\lambda}{b}\right) \right|\, d\mu_k(\lambda).
\end{eqnarray*}
Applying the Cauchy-Schwarz inequality, we obtain that
\begin{eqnarray*}
\mathop{\underset{ x\in \mathbb{R}}{\sup}}\, \left|\frac{d^{p}}{dx^{p}}\left( (1+x^2)^{q}\,\phi\right)\right| \lesssim\frac{1}{|b|^{2k+2}} \int_{\mathbb{R}} \left(1+|\lambda|^2\right)^{4l}\left|\mathcal{D}_k((1+x^2)^q\phi)\left(\frac{\lambda}{b}\right)\right|^2\,d\mu_k(\lambda).
\end{eqnarray*}
From \eqref{e3.1}, we derive that
\begin{equation*}
=   \frac{1}{|b|^{2k+2}} \,\int_{\mathbb{R}}\left|\mathcal{D}_k\left((I - \Lambda_k^2)^{2l} (1+x^2)^q\,\phi \right)\left(\frac{\lambda}{b} \right) \right|^2 \, d\mu_k(\lambda),
\end{equation*}
where $I$ denotes the identity operator.
Using the Plancherel's formula, Theorem 3.2(ii), and Proposition 2.2(iii) from \cite{Chettaoui}, and subsequently applying Remark \ref{r:3.1}(ii), we deduce that
\begin{eqnarray*}
\mathop{\underset{ x\in \mathbb{R}}{\sup}}\, \left|\frac{d^{p}}{dx^{p}}\left( (1+x^2)^{q}\,\phi\right)\right| \lesssim \left( \int_{\mathbb{R}} (1+x^2)^{2m}|\Lambda_{k, M^{-1}}^n \varphi(y)|^2\,d\mu_k(y)\right)^{\frac{1}{2}},   
\end{eqnarray*}
where  $m,n\in \mathbb{N}^*$. Thus,
\begin{equation} \label{e:2.10}
  S_{p,q}(\phi) \lesssim \|  (1+x^2)^m\,\Lambda_{k, M^{-1}}^{n} \varphi \|_{L_k^2(\mathbb{R})}.
\end{equation}
By combining  \eqref{e:2.9} and \eqref{e:2.10}, we conclude that the family of seminorms $\{ S_{r,p}:r,p\in \mathbb{N}\}$ generates the topology on $\mathcal{S}(\mathbb{R})$.
\end{proof}
\section{Paley-Wiener theorem for \texorpdfstring{$\mathcal{S}(\mathbb{R})$}{\mathcal{S}(\mathbb{R})}} \label{S3}
\par  In 2002, Trim\'eche established the Paley-Wiener theorems for the Dunkl transform and Dunkl translation operators by utilizing results from Dunkl Sobolev spaces. In addition, Chettaoui and Trim\'eche investigated a new type of Paley-Wiener theorem for the Dunkl transform on 
$\mathbb{R}$ \cite{Chettaoui, Trime'Che}. Following the proof techniques of Trim\'eche, we establish the real Paley-Wiener theorem for the linear canonical Dunkl transform.
\begin{theorem} \label{t:4.1}
For all $f \in \mathcal{ S(\mathbb{R})} $, the following limit exists
\begin{equation*}
    \lim_{n \rightarrow \infty} \|  \Lambda_{k, M^{-1}}^nf\|^{\frac{1}{n}}_{L_k^p(\mathbb{R})} = \sigma_f, \quad 1\le p \le \infty,
\end{equation*}
where
\begin{equation} \label{e4.1}
\sigma_f = \sup \,\left\{ \Big|\frac{\lambda}{b}\Big|:  \lambda \in\text{supp}\,(\mathcal{D}_k^M(f)
)\right\}.
\end{equation}
\end{theorem}
\begin{proof}
Suppose $\sigma_f = 0$. This implies $ \text{supp}\,(\mathcal{D}_k^M(f))=\{0\}$, and for every non zero $\lambda\in \mathbb{R}$, we have $\mathcal{D}_k^M(f)(\lambda) = 0.$
From \eqref{e:2.2}, we deduce  $f=0$ a.e. Thus, 
\begin{equation*}
    \lim_{n \rightarrow \infty} \| \Lambda_{k, M^{-1}}^nf\|^{\frac{1}{n}}_{L_k^p(\mathbb{R})} = 0.
\end{equation*}
First, we establish the result for the case $p = 2$, while assuming $0 < \sigma_f \leq \infty $.  
It is evident that for all $n \in \mathbb{N}^* $, $\Lambda_{k, M^{-1}}^n f \in \mathcal{S}(\mathbb{R})$, for $f \in \mathcal{S}(\mathbb{R})$. Moreover, it is known that the operator $\mathcal{D}_k^M $ is a topological isomorphism from $ \mathcal{S}(\mathbb{R})$ onto itself. Thus, by using \eqref{eq :2.1} and \eqref{e:2.3}, we obtain the following:
\begin{eqnarray*}
\|\Lambda_{k, M^{-1}}^n (f) \|^2_{L_k^2(\mathbb{R})}& = &\| \mathcal{D}_k^M(\Lambda_{k, M^{-1}}^n (f) )\|^2_{L_k^2(\mathbb{R})}\\
&=& \int_{\mathbb{R}} \left| \frac{\lambda}{b} \right|^{2n}\, |\mathcal{D}_k^M(f)(\lambda)|^2\, d\mu_k(\lambda).
\end{eqnarray*}
If  $\sigma_f = \infty$, then for every $N \in \mathbb{N}$ we have
\begin{eqnarray*}
\int_{\mathbb{R}} \left| \frac{\lambda}{b} \right|^{2n}\,|\mathcal{D}_k^M(f)(\lambda)|^2\,d\mu_k(\lambda) &\ge& \int_N^ \infty  \left| \frac{\lambda}{b} \right|^{2n}\,|\mathcal{D}_k^M(f)(\lambda)|^2\,d\mu_k(\lambda) \\
&\ge& \left| \frac{N}{b} \right|^{2n}\,\int_N^ \infty |\mathcal{D}_k^M(f)(\lambda)|^2 \,d\mu_k(\lambda)\\
\|\Lambda_{k, M^{-1}}^n(f)\|_{L_k^2(\mathbb{R})}^ {\frac{1}{n}} &\ge&  \frac{N}{|b|} \left( \int_N^ \infty |\mathcal{D}_k^M(f)(\lambda)|^2 \,d\mu_k(\lambda)\right)^{\frac{1}{2n}}\\
\lim_{n \rightarrow \infty} \|\Lambda_{k, M^{-1}}^n(f)\|_{L_k^2(\mathbb{R})}^ {\frac{1}{n}} &\ge& \frac{N}{|b|}\lim_{n \rightarrow \infty} \left( \int_N^ \infty |\mathcal{D}_k^M(f)(\lambda)|^2 \,d\mu_k(\lambda)\right)^{\frac{1}{2n}}\\
\lim_{n \rightarrow \infty} \|\Lambda_{k, M^{-1}}^n(f)\|_{L_k^2(\mathbb{R})}^ {\frac{1}{n}} &\ge &  \frac{N}{|b|}, \quad \forall\,\, N\in \mathbb{N}.
\end{eqnarray*}
It is clear that 
\begin{equation*}
\lim_{n \rightarrow \infty} \|\Lambda_{k, M^{-1}}^n(f)\|_{L_k^2(\mathbb{R})}^ {\frac{1}{n}} = \infty.    
\end{equation*}
If $\sigma_f \in (0, \infty)$, then
\begin{eqnarray*}
\|\Lambda_{k, M^{-1}}^n(f)\|^2_{L_k^2(\mathbb{R})} &=& \int_{\mathbb{R}}  \left| \frac{\lambda}{b} \right|^{2n}\,|\mathcal{D}_k^M(f)(\lambda)|^2\, d\mu_k(\lambda)\\
&\le& \sigma_f^{2n} \int_{\mathbb{R}} |\mathcal{D}_k^M(f)(\lambda)|^2\,d\mu_k(\lambda)\\
&=&  \sigma_f^{2n} \, \|\mathcal{D}_k^M(f)\|^2_{L_k^2(\mathbb{R})}\\
\|\Lambda_{k, M^{-1}}^n(f)\|_{L_k^2(\mathbb{R})}^{\frac{1}{n}}&\le& \sigma_f \|f\|^{\frac{1}{n}}_{L_k^2(\mathbb{R})}.
\end{eqnarray*}
Taking limit on both sides we deduce that
\begin{equation} \label{e:3.1}
\lim_{n \rightarrow \infty}\, \text{sup}\, \|\Lambda_{k, M^{-1}}^n(f)\|_{L_k^2(\mathbb{R})}^{\frac{1}{n}} \le \sigma_f.
\end{equation}
On the other hand, let $\epsilon>0$ such that $0<\epsilon < \sigma_f$. 
\begin{eqnarray}
\nonumber\|\Lambda_{k, M^{-1}}^n(f)\|^2_{L_k^2(\mathbb{R})} &=& \int_0^\infty  \left| \frac{\lambda}{b} \right|^{2n}\, \left( |\mathcal{D}_k^M(f)(\lambda)|^2+|\mathcal{D}_k^M(f)(-\lambda)|^2\right)\, d\mu_k(\lambda)\\
\nonumber&\ge& \int_{|b|(\sigma_f-\epsilon)}^{|b|\sigma_f}  \left| \frac{\lambda}{b} \right|^{2n}\,\left( |\mathcal{D}_k^M(f)(\lambda)|^2+|\mathcal{D}_k^M(f)(-\lambda)|^2\right)\, d\mu_k(\lambda)\\
\nonumber&\ge& (\sigma_f-\epsilon)^{2n} \int_{\sigma_f -\epsilon}^{\sigma_f}\left( |\mathcal{D}_k^M(f)(\lambda)|^2+|\mathcal{D}_k^M(f)(-\lambda)|^2\right)\, d\mu_k(\lambda)\\
\lim_{n \rightarrow \infty}\, \text{inf}\, \|\Lambda_{k, M^{-1}}^n(f)\|^{\frac{1}{n}}_{L_k^2(\mathbb{R})} &\ge& \sigma_f - \epsilon. \label{e:3.2}
\end{eqnarray}
By combining \eqref{e:3.1} and \eqref{e:3.2}, we conclude that 
\begin{equation} \label{e:3.3}
 \lim_{n \rightarrow \infty}\,  \|\Lambda_{k, M^{-1}}^n(f)\|_{L_k^2(\mathbb{R})}^{\frac{1}{n}} = \sigma_f.   
\end{equation}
Next, we vary $p\in(2, \infty)$ and $\sigma_f \in (0, \infty)$. In this case, we first prove the following inequality
\begin{equation*}
    \lim_{n \rightarrow \infty}\,  \|\Lambda_{k, M^{-1}}^n(f)\|_{L_k^p(\mathbb{R})}^{\frac{1}{n}} \le \sigma_f.
\end{equation*}
Using \eqref{e:2.3} and \eqref{e:2.2}, we obtain that $\Lambda_{k, M^{-1}}^n(f) = \mathcal{D}_k^{M^{-1}}\left( (\frac{i\lambda}{b})^n\,  \mathcal{D}_k^M(f)\right)$. Then applying Young's inequality \eqref{eq:2.3} we have 
\begin{eqnarray}
\nonumber\|\Lambda_{k, M^{-1}}^n(f)\|_{L_k^p(\mathbb{R})} &=&  \| \mathcal{D}_k^{M^{-1}}( (i\lambda/b)^n \mathcal{D}_k^Mf) \|_{L_k^p(\mathbb{R})} \\
\nonumber&\le& C_{b,p} \,\| (i\lambda/b)^n\, \mathcal{D}_k^M(f) \|_{L_k^q(\mathbb{R})}\\
\nonumber&\le& C_{b,p} \,\left(\int_{-|b|\sigma_f}^{|b|\sigma_f} |\lambda/b|^{qn}\,|\mathcal{D}_k^M(f)(\lambda)|^q\, d\mu_k(\lambda)\right)^{\frac{1}{q}}\\
\nonumber&\le& C_{b,p}\, \sigma_f^n\, \|\mathcal{D}_k^M(f)\|_{L_k^q(\mathbb{R})}.
\end{eqnarray}
Thus,
\begin{equation}
\lim_{n \rightarrow \infty}\,  \|\Lambda_{k, M^{-1}}^n(f)\|_{L_k^p(\mathbb{R})}^{\frac{1}{n}} \le \sigma_f. \label{e:3.4} 
\end{equation}
Now, we take $p\in [1,2)$ and 
  $\sigma_f \in (0, \infty)$. In this case, we prove that
\begin{equation*}
    \lim_{n \rightarrow \infty}\, \text{sup}\, \|\Lambda_{k, M^{-1}}^n(f)\|_{L_k^p(\mathbb{R})}^{\frac{1}{n}} \le \sigma_f.
\end{equation*}
Let us consider
\begin{eqnarray*}
\|f\|_{L_k^p(\mathbb{R})}^p &=& \int_{\mathbb{R}} |f(\lambda)|^p\, d\mu_k(\lambda)\\
&=& \int_{\mathbb{R}} (1+\lambda^2)^{-lp}\,(1+\lambda^2)^{lp}\,|f(\lambda)|^p\, d\mu_k(\lambda)
\end{eqnarray*}
where $l\in \mathbb{N}$ such that $l>(k+1) (\frac{2}{p}-1)$. By applying H\"older's inequality, we obtain that
\begin{eqnarray} \label{eq:3.3}
    \|f\|^p_{L_k^p(\mathbb{R})} &\le& \|(1+\lambda^2)^{-lp}\|_{L_k^{\frac{2}{2-p}}(\mathbb{R})}\, \|(1+\lambda^2)^l\,f\|^p_{L_k^2(\mathbb{R})}.
\\
\nonumber\|f\|_{L_k^p(\mathbb{R})}&\lesssim& \|(1+\lambda^2)^l\,f\| _{L_k^2(\mathbb{R})}.    
\end{eqnarray}
From \eqref{e:2.2}, we have 
\begin{eqnarray*}
 \|(1+\lambda^2)^l\,f\| _{L_k^2(\mathbb{R})} = \| (1+\lambda^2)^l\,\mathcal{D}_k^{M^{-1}}(\mathcal{D}_k^Mf)\| _{L_k^2(\mathbb{R})}  
\end{eqnarray*}
Since $f \in \mathcal{S}(\mathbb{R})$ and  applying Theorem \ref{t3.8}  imply that there exist $\beta, \gamma \in \mathbb{N}^*$ such that
\begin{equation}\label{eq:3.4}
\|(1+\lambda^2)^l\,f\|_{L_k^2(\mathbb{R})} \lesssim  \|(1+\lambda^2)^\beta
    \,\Lambda_{k, M^{-1}}^\gamma (\mathcal{D}_k^M f)\|_{L_k^2(\mathbb{R})}.
\end{equation}
Applying equation \eqref{eq:3.4} in \eqref{eq:3.3} becomes
\begin{equation}\label{eq:3.5}
   \|f\|^p_{L_k^p(\mathbb{R})} \lesssim  \,\|(1+\lambda^2)^\beta
    \,\Lambda_{k, M^{-1}}^\gamma( \mathcal{D}_k^Mf)\|_{L_k^2(\mathbb{R})}. 
\end{equation}
 From \eqref{eq:3.5} and \eqref{e:2.3} we deduce that 
\begin{eqnarray*}
\|\Lambda_{k, M^{-1}}^n(f)\|_{L_k^p(\mathbb{R})} &\lesssim&  \|(1+\lambda^2)^\beta
    \, \Lambda_{k, M^{-1}}^\gamma(\mathcal{D}_k^M(\Lambda_{k, M^{-1}}^nf))\|_{L_k^2(\mathbb{R})} \\
    &\lesssim& \|(1+\lambda^2)^\beta
    \,\Lambda_{k, M^{-1}}^\gamma( 
( i\lambda/b)^n \mathcal{D}_k^Mf)\|_{L_k^2(\mathbb{R})}.
\end{eqnarray*}
Using  support of the $\mathcal{D}_k^M$ of Schwartz function $f$ and Proposition \ref{p:2.1}, we obtain that
\begin{equation*}
 \|\Lambda_{k, M^{-1}}^n(f)\|_{L_k^p(\mathbb{R})} \lesssim   \sigma_f.
\end{equation*}
 Therefore we obtain that
\begin{equation} \label{e:3.8}
\lim_{n \rightarrow \infty}\, \text{sup}\, \|\Lambda_{k, M^{-1}}^n(f)\|_{L_k^p(\mathbb{R})}^{\frac{1}{n}} \le \sigma_f.    
\end{equation}
On the other hand side, for $1 \le p < \infty$ and $0 <\sigma_f \le \infty$, we have to prove 
\begin{equation*}
 \lim_{n \rightarrow \infty}\, \|\Lambda_{k, M^{-1}}^n(f)\|_{L_k^p(\mathbb{R})}^{\frac{1}{n}} \ge \sigma_f.   
\end{equation*}
Invoking \eqref{eq:2.4} and H\"older's inequality, we have 
\begin{eqnarray}
 \nonumber\int_{\mathbb{R}} |\Lambda_{k, M^{-1}}^n(f)(\lambda)|^2\, d\mu_k(\lambda) &= &(-1)^n\int_{\mathbb{R}} \Lambda_{k, M^{-1}}^{2n}(f)(\lambda)\, f(\lambda)\, d\mu_k(\lambda)\\
&\le& \|f\|_{L_k^q(\mathbb{R})}\|\Lambda_{k, M^{-1}}^{2n}(f)\|_{L_k^p(\mathbb{R})}.\label{e:4.10}
\end{eqnarray}
Utilizing \eqref{e:3.3} and \eqref{e:4.10}, we deduce the following inequalities
\begin{eqnarray}
\nonumber\sigma_f &=&  \lim_{n \rightarrow \infty}\, \|\Lambda_{k, M^{-1}}^n(f)\|^{\frac{1}{n}}_{L_k^2(\mathbb{R})}\\
\nonumber &\le&  \lim_{n \rightarrow \infty}\,\text{inf}\, \|\Lambda_{k, M^{-1}}^{2n}(f)\|^{\frac{1}{2n}}_{L_k^p(\mathbb{R})}\,  \lim_{n \rightarrow \infty}\,\text{inf}\, \|f\|^{\frac{1}{2n}}_{L_k^q(\mathbb{R})}\\
&\le& \lim_{n \rightarrow \infty}\,\text{inf}\, \|\Lambda_{k, M^{-1}}^{2n}(f)\|^{\frac{1}{2n}}_{L_k^p(\mathbb{R})} .\label{eq:3.7}
\end{eqnarray}
Similarly, for every $n\in \mathbb{N}^*$, we have
\begin{eqnarray}\label{e4.10}
\|\Lambda_{k, M^{-1}}^{n+1}(f)\|^2_{L_k^2(\mathbb{R})} \le \| \Lambda_{k, M^{-1}}(f)\|_{L_k^q(\mathbb{R})}\| \Lambda_{k, M^{-1}}^{2n+1}(f)\|_{L_k^p(\mathbb{R})},
\end{eqnarray}
where $p$ and $q$ are the conjugate exponents. From \eqref{e:3.3}, we deduce
\begin{eqnarray*}
    \sigma_f = \lim_{n \rightarrow \infty}\, \|\Lambda_{k, M^{-1}}^{n+1}(f)\|^{\frac{1}{n+1}}_{L_k^2(\mathbb{R})}
   =\lim_{n \rightarrow \infty}\, \|\Lambda_{k, M^{-1}}^{n+1}(f)\|^{\frac{1}{n+1/2}}_{L_k^2(\mathbb{R})}. 
\end{eqnarray*}
Using the inequality \eqref{e4.10} in the above equation, we obtain that
\begin{equation} \label{eq:3.8}
\sigma_f \le \lim_{n \rightarrow \infty}\,\text{inf}\, \|\Lambda_{k, M^{-1}}^{2n+1}(f)\|^{\frac{1}{2n+1}}_{L_k^p(\mathbb{R})}.
\end{equation}
From  \eqref{eq:3.7} and \eqref{eq:3.8} we deduce that
\begin{equation}\label{eq:3.9}
    \sigma_f \le \lim_{n \rightarrow \infty}\,\text{inf}\, \|\Lambda_{k, M^{-1}}^{n}(f)\|^{\frac{1}{n}}_{L_k^p(\mathbb{R})}.
\end{equation}
By combining the equations \eqref{e:3.4}, \eqref{e:3.8} and \eqref{eq:3.9} we conclude that
\begin{equation*}
 \lim_{n \rightarrow \infty} \| \Lambda_{k, M^{-1}}^nf\|^{\frac{1}{n}}_{L_k^p(\mathbb{R})} = \sigma_f, \qquad \text{where} \,\,\, \sigma_f \in(0, \infty).  
\end{equation*}
If $\sigma_f = \infty$, then it is evident from  \eqref{eq:3.9} that
\begin{equation*}
     \lim_{n \rightarrow \infty} \| \Lambda_{k, M^{-1}}^nf\|^{\frac{1}{n}}_{L_k^p(\mathbb{R})} = \infty.
\end{equation*}
\end{proof}
\section{Real Paley-Wiener type theorems} \label{S5}
Apart from the classical real  Paley-Wiener theorem, over time the researchers have gradually proved the real Paley-Wiener theorem by using various techniques like functions with symmetric body spectrum, compact spectrum, and 
the polynomial domain, etc. In this section, we explore two versions of the real Paley-Wiener theorem for the linear canonical Dunkl transform. The first one is derived using the concept of the polynomial domain, and then we establish the Boas-type theorem for the linear canonical Dunkl transform.

 \subsection{The linear canonical Dunkl transform of functions with polynomial domain support}

Tuan studied the real Paley-Wiener theorem for a differential operator in the context of the classical Fourier transform, where the transform of a function is supported in the polynomial domain \cite{V.K. Tuan}. Following the proof techniques of Tuan's approach, we investigate the real Paley-Wiener theorem for the linear canonical Dunkl operator.
\begin{definition} \label{d5.1}
Let $b\neq 0$ and $P(x)$ be a non-constant polynomial. Then $\Omega_{P,b} :=\{ \lambda \in \mathbb{R}: |P\left(\frac{\lambda}{b}\right)| \le 1 \}$ is called polynomial domain.    
\end{definition}
 
 \begin{theorem} \label{t:5.1}
 The linear canonical Dunkl transform $\mathcal{D}_k^M(f)$ of $f\in \mathcal{S}(\mathbb{R})$   vanishes outside of the polynomial domain $\Omega_{P,b}$ if and only if  
 \begin{equation} \label{e:3.12}
 \limsup\limits_{n \to \infty} \|P^n(i\Lambda_{k, M^{-1}})f\|^{\frac{1}{n}}_{L_k^p(\mathbb{R})} \le 1, \quad \text{
 for}\quad 1\le p \le \infty.
 \end{equation}
 \end{theorem}
 \begin{proof}
 The proof will be divided into two cases. We proceed with the first case by assuming $p$ varies from 1 to 2. Suppose $f\in  \mathcal{S}(\mathbb{R})$, the inequality \eqref{e:3.12}  holds. Let us prove  $$\text{supp}\, \mathcal{D}_k^Mf \subset \Omega_{P,b}.$$ We begin the proof by assuming $\lambda \notin\Omega_{P,b}$. Using the relation \eqref{e:2.3} and Hausdorff-Young’s inequality, we obtain the following: 
 \begin{eqnarray}
\nonumber\left\|P^n(\lambda/b)\,\mathcal{D}_k^M(f) \right\|_{L_k^q(\mathbb{R})} &=& \|\mathcal{D}_k^M(P^n(i\Lambda_{k, M^{-1}})f)\|_{L_k^q(\mathbb{R})}\\
 \nonumber&\le&   \frac{1}{\left(|b|^{k+1}\right)^{1-\frac{2}{q}}} \,\| P^n(i\Lambda_{k, M^{-1}})f\|_{L_k^p(\mathbb{R})}  \\
 \limsup\limits_{n \to \infty}\left\|P^n\left(\lambda/b\right)\,\mathcal{D}_k^M(f) \right\|^{\frac{1}{n}}_{L_k^q(\mathbb{R})} &\le&  1. \label{e:3.13}
 \end{eqnarray}
 Fix $\lambda_0 \notin \Omega_{P,b}$. By Definition \ref{d5.1}, it is clear that $P\left(\lambda_0/b\right)\ge1$, for some neighborhood around $\lambda_0$ (say  $B_{\lambda_0}$). Therefore, for any $\lambda \in  B_{\lambda_0}$, we have 
 \begin{equation} \label{e:3.14}
 |P(\lambda/b)| \ge \frac{1+|P\left(\lambda_0/b\right)|}{2}. 
 \end{equation}
 We proceed the calculation by using  \eqref{e:3.13} and \eqref{e:3.14} in the following manner:
 \begin{eqnarray*}
    1&\ge& \limsup\limits_{n \to \infty}\left\|P^n\left(\lambda/b\right)\,\mathcal{D}_k^M(f) \right\|^{\frac{1}{n}}_{L_k^q(\mathbb{R})} \\
    &\ge&\limsup\limits_{n \to \infty} \left(\int_{B_{{\lambda}_0}} |P^n\left(\lambda/b\right)|^q \,|\mathcal{D}_k^M(f)|^q\,d\mu_k(\lambda)\right)^{\frac{1}{nq}}\\
    &\ge& \frac{1+P\left(\lambda_0/b\right)}{2} \limsup\limits_{n \to \infty} \left(\int_{B_{{\lambda}_0}} |\mathcal{D}_k^M(f)|^q\,d\mu_k(\lambda)\right)^{\frac{1}{nq}}.
 \end{eqnarray*}
 It is clear from \eqref{e:3.14} that $\frac{1+P\left(\lambda_0/b\right)}{2} >1 $. Therefore,
\begin{eqnarray*}
\limsup\limits_{n \to \infty} \left(\int_{B_{{\lambda}_0}} |\mathcal{D}_k^M(f)(\lambda)|^q\,d\mu_k(\lambda)\right)^{\frac{1}{nq}} =0.
 \end{eqnarray*}
 This implies $\lambda_0 \notin \text{supp}\, \mathcal{D}_k^Mf $. Conversely,  assume that supp$\mathcal{D}_k^M(f) \subset\Omega_{P,b}$. For suitable $p$, we consider
 \begin{eqnarray*}
     \|f\|^p_{L_k^p(\mathbb{R})} = \int_{\mathbb{R}}\,(1+|\lambda/b|)^{-p}\,(1+|\lambda/b|)^p|f|^{p}\,d\mu_k(\lambda).
 \end{eqnarray*}
Using H\"older's inequality to the above identity, we deduce that
\begin{eqnarray*}
\|f\|^p_{L_k^p(\mathbb{R})} &\
\le& \|(1+|\lambda/b|)^{-p}\|_{\frac{2}{2-p}}\|(1+|\lambda/b|)f\|^p_{L_k^2(\mathbb{R})}  \\
&=& C_b\,\|(1+\lambda/b)f\|^p_{L_k^2(\mathbb{R})}. 
\end{eqnarray*}
We now apply the above argument for 
$P^n(i\Lambda_{k, M^{-1}})f$ and invoking \eqref{e:2.3}, we have
\begin{eqnarray*}
 \|P^n(i\Lambda_{k, M^{-1}})f\|_{L_k^p(\mathbb{R})} &\le& C_b^{\frac{1}{p}}\,\|(1+\lambda/b)P^n(i\Lambda_{k, M^{-1}})f\|_{L_k^2(\mathbb{R})}\\
  \|P^n(i\Lambda_{k, M^{-1}})f\|_{L_k^p(\mathbb{R})} &\le& C_b^{\frac{1}{p}}\,\|(I+ i\Lambda_{k, M^{-1}})\,P^n(\lambda/b)\,\mathcal{D}_k(f)\|_{L_k^2(\mathbb{R})}.
\end{eqnarray*}
Exploiting the definition of $\Lambda_{k, M^{-1}}$ and the assumption on the support of $\mathcal{D}_k^M(f)$, we obtain the following inequality:
\begin{eqnarray*}
    \|P^n(i\Lambda_{k, M^{-1}})f\|_{L_k^p(\mathbb{R})}\le n\, C(a,b)\,\|P^{n-1} (\lambda/b)\, \Phi_n\|_{L_k^2(\mathbb{R})} , 
\end{eqnarray*}
where supp $\Phi_n \subset$ supp $\mathcal{D}_k^M(f)$, and  $C(a,b)$ is a constant.
There exists a suitable constant $c'$, independent of $n$,  such that 
\begin{eqnarray*}
     \|P^n(i\Lambda_{k, M^{-1}})f\|_{L_k^p(\mathbb{R})}^{\frac{1}{n}}\le (n\,  C(a,b)\,c')^{\frac{1}{n}}.
\end{eqnarray*}
Taking the limit supremum as 
$n \to \infty$ we conclude that
\begin{eqnarray*}
  \limsup\limits_{n \to \infty}   \|P^n(i\Lambda_{k, M^{-1}})f\|_{L_k^p(\mathbb{R})}^{\frac{1}{n}}\le 1.
\end{eqnarray*}
In the second part of the proof, we consider the case $2<p \le \infty$ and assume that  \eqref{e:3.12} holds.  Doing some manipulation and applying the H\"older's inequality we derive the following inequality:
\begin{eqnarray*}
\|P^n(i\Lambda_{k, M^{-1}})f\|^2_{L_k^2(\mathbb{R})} &=&  \int_{\mathbb{R}} P^n(i\Lambda_{k, M^{-1}})f (x)\, P^n(i\Lambda_{k, M^{-1}}) \overline{f(x)}\, d\mu_k(x)\\
&=& \int_{\mathbb{R}} f(x)\, P^{2n}(i\Lambda_{k, M^{-1}})\overline{f(x)}\,d\mu_k(x)\\
&\le& \|f\|_{L_k^q(\mathbb{R})} \|P^{2n}(i\Lambda_{k, M^{-1}})f\|_{L_k^p(\mathbb{R})},
\end{eqnarray*}
where $q$ is the conjugate exponent of $p$. Taking the limit supremum on both sides, we have
\begin{eqnarray*}
 \limsup\limits_{n \to \infty} \|P^n(i\Lambda_{k, M^{-1}})f\|_{L_k^2(\mathbb{R})}^{\frac{1}{n}} &\le&  \limsup\limits_{n \to \infty}  \|f\|_{L_k^q(\mathbb{R})}^{\frac{1}{2n}}\, \|P^{2n}(i\Lambda_{k, M^{-1}})f\|_{L_k^p(\mathbb{R})}^{\frac{1}{2n}}\\
 &\le& 1.
 \end{eqnarray*}
For $p=2$, it is evident from the first part of the proof that the support of  $\mathcal{D}_k^M(f)$ is contained in $\Omega_{P,b}$.
 \newline Conversely let us assume that supp $\mathcal{D}_k^M(f) \subset \Omega_{P,b}$. We consider
 \begin{eqnarray*}
    \|P^n(i\Lambda_{k, M^{-1}})\mathcal{D}_k^{M^{-1}}(\mathcal{D}_k^Mf)\|_{L_k^p(\mathbb{R})}&=& \| \mathcal{D}_k^{M^{-1}}(P^n(\lambda/b)\,\mathcal{D}_k^{M}(f))\|_{L_k^p(\mathbb{R})}.
\end{eqnarray*}
Applying Hausdorff-Young's inequality \eqref{eq:2.3}, we will obtain the required result
\begin{eqnarray*}
  \|P^n(i\Lambda_{k, M^{-1}})f\|_{L_k^p(\mathbb{R})}  &\le& \frac{1}{\left(|b|^{k+1}\right)^{1-\frac{2}{p}}}\, \|P^n(\lambda/b)\,\mathcal{D}_k^{M}(f)\|_{L^q_k(\mathbb{R})}\\
 \limsup\limits_{n \to \infty}\|P^n(i\Lambda_{k, M^{-1}})f\|_{L_k^p(\mathbb{R})} ^{\frac{1}{n}} &\le&  \limsup\limits_{n \to \infty} \left( \frac{1}{\left(|b|^{k+1}\right)^{1-\frac{2}{p}}}\right)^{\frac{1}{n}}\, \|\mathcal{D}_k^{M}(f)\|_{L^q_k(\mathbb{R})}^{\frac{1}{n}}\\
 &\le& 1.
 \end{eqnarray*}
This completes the proof. 
\end{proof}
The above theorem characterizes a function $f$ whose linear canonical Dunkl transform is supported within the polynomial domain. Setting $P(x) = x^2$, yields $P(i\Lambda_{k, M^{-1}}) = - \Delta_{k,M^{-1}}$, 
which leads to the following corollary.  Moreover, this corollary is a real version of the Paley-Wiener theorem for the linear canonical Dunkl Laplacian. 
\begin{corollary}
Let $f \in \mathcal{S}(\mathbb{R})$ and   support of $\mathcal{D}_k^M(f)$ is  compact if and only if 
\begin{equation*}
   \lim_{n \to \infty} \|  \Delta_{k,M^{-1}}^nf\|^{\frac{1}{n}}_{L_k^p(\mathbb{R})} < \infty.
\end{equation*}
\end{corollary}
\subsection{Boas-type Paley-Wiener theorem}
Boas initially characterized $L^2$
  functions that vanish on a closed interval \cite{Boas}. Later, Tuan and Zayed establish Boas's theorem for a class of integral transforms \cite{Zayed}. We extended the Boas-type Paley-Wiener theorem for the linear canonical Dunkl transform. To facilitate its proof, we first recall the following lemma.
\begin{lemma} \label{l6.1}
If $g$ belongs to $L^p(U, m)$ where $U\subset \mathbb{R}$ and  $m$ is a Lebesgue measure on $U$ with $m(U) =1$, then 
\begin{equation*}
\lim_{p \to \infty} \|g\|_{L^p(U,\,m)} = \|g\|_{L^\infty(U,\,m)}.
\end{equation*}
\end{lemma}
The forthcoming theorem may be proved in much the same way as Theorem \ref{t:5.1}.
\begin{theorem} \label{t:5.5}
Let $f \in \mathcal{S}(\mathbb{R})$  such that  $\mathcal{D}_k^M(f)$ vanishes outside the interval $(-\sigma_f, \sigma_f)$ and $1 \le p \le \infty$. Then 
\begin{equation} \label{e:6.1}
 \lim_{n \to \infty}  \Bigg \| \sum_{m=0}^\infty \frac{n^m\, \Delta^m_{k,M^{-1}}f}{m!} 
 \Bigg \|_{L_k^p(\mathbb{R})}^{\frac{1}{2n}}= e^{-\delta_f},
\end{equation}
where  
$$ \delta_f = \inf \{|\lambda/b|^2: \lambda \in \text{supp}\,\mathcal{D}_k^Mf\}$$ and  $\sigma_f$ defined in \eqref{e4.1}. 
\end{theorem}
\begin{proof}
We begin by proving the case $p=2$.
Let $f$ be a non-zero function in $\mathcal{S}(\mathbb{R})$ with supp $\mathcal{D}_k^M(f) \subset (-\sigma_f, \sigma_f)$. We consider
\begin{eqnarray*}
\Bigg \| \sum_{m=0}^\infty \frac{n^m\, \Delta^m_{k,M^{-1}}f}{m!} 
 \Bigg \|_{L_k^2(\mathbb{R})} = \Bigg \| \mathcal{D}_k^M \left(\sum_{m=0}^\infty \frac{n^m\, \Delta^m_{k,M^{-1}}f}{m!} \right)
 \Bigg \|_{L_k^2(\mathbb{R})}.
\end{eqnarray*}
Using \eqref{e:3.6} and taking $\frac{1}{2n}$ th power on both sides, we obtain the following 
\begin{eqnarray} \label{e6.2}
\Bigg \| \sum_{m=0}^\infty \frac{n^m\, \Delta^m_{k,M^{-1}}f}{m!} 
 \Bigg \|^2_{L_k^2(\mathbb{R})} & = & \int_{\mathbb{R}} |e^{-n\lambda^2/b^2}|^2\, |\mathcal{D}_k^M(f)|^2\, d\mu_k(\lambda)\\
\nonumber &=& \int_{\text{supp}\,\mathcal{D}_k^M(f)}e^{-2n\lambda^2/b^2}\,|\mathcal{D}_k^M(f)|^2\, d\mu_k(\lambda)\\
 \Bigg \| \sum_{m=0}^\infty \frac{n^m\, \Delta^m_{k,M^{-1}}f}{m!} 
 \Bigg \|^{\frac{1}{n}}_{L_k^2(\mathbb{R})} 
 \nonumber&=& \|f\|^{\frac{1}{2n}}_{L_k^2(\mathbb{R})} \left(  \int_{\text{supp}\,\mathcal{D}_k^M(f)}e^{-2n\lambda^2/b^2}\,|\mathcal{D}_k^M(f)|^2\, \frac{d\mu_k(\lambda)}{\|f\|_{L_k^2(\mathbb{R})}}\right)^{\frac{1}{2n}}.
\end{eqnarray}
If we take $g(\lambda) = e^{-\lambda^2/b^2}$,  $U= \text{supp}\,\mathcal{D}_k^M(f)$, $p=2n$, and the measure
$dm(\lambda) = \|f\|^{-1}_{L_k^2(\mathbb{R})}\, |\mathcal{D}_k^M(f)(\lambda)|^2\,d\mu_k(\lambda)$, then $ \int_{U} dm(\lambda) =1$. 
 Now, applying Lemma \ref{l6.1} based on the above ingredients, we deduce that
\begin{eqnarray}
\nonumber \lim_{n \to \infty} \Bigg \| \sum_{m=0}^\infty \frac{n^m\, \Delta^m_{k,M^{-1}}f}{m!} 
 \Bigg \|^{\frac{1}{n}}_{L_k^2(\mathbb{R})} &=& \sup_{\lambda \in \text{supp}\,\mathcal{D}_k^M(f)}e^{-|\lambda/b|^2} \lim_{n \to \infty} \|f\|^{\frac{1}{2n}}_{L_k^2(\mathbb{R})}\\
 &=& e^{-\delta_f}\label{e:6.2}.
\end{eqnarray}

Hence, the required result is proved for $p=2$.
By utilizing  \eqref{e:6.2}, we will prove \eqref{e:6.1} for other $p$. We first concentrate on proving the one-way inequality
\begin{equation*}
    \liminf_{n \to \infty}\Bigg \| \sum_{m=0}^\infty \frac{n^m\, \Delta^m_{k,M^{-1}}f}{m!} 
 \Bigg \|^{\frac{1}{2n}}_{L_k^p(\mathbb{R})} \ge e^{-\delta_f}, \quad 1 \le p \le \infty.
\end{equation*}
Let us consider \eqref{e6.2} and further using Parsavel formula \eqref{eq3.3} and  H\"older's inequality, we deduce that 
\begin{eqnarray*}
 \Bigg \| \sum_{m=0}^\infty \frac{n^m\, \Delta^m_{k,M^{-1}}f}{m!} 
 \Bigg \|^2_{L_k^2(\mathbb{R})} & = & \int_{\mathbb{R}} |e^{-n\lambda^2/b^2}|^2\, |\mathcal{D}_k^M(f)|^2\, d\mu_k(\lambda)\\
 &=& \int_{\mathbb{R}}e^{-2n\lambda^2/b^2}\,\mathcal{D}_k^M(f)(\lambda)\, \overline{\mathcal{D}_k^M(f)(\lambda)}\, d\mu_k(\lambda)\\
 &=& \int_{\mathbb{R}}\mathcal{D}_k^M \left( \sum_{m=0}^\infty \frac{(2n)^m\, \Delta^m_{k,M^{-1}}f}{m!} \right) (\lambda)\, \overline{\mathcal{D}_k^M(f)(\lambda)}\,d\mu_k(\lambda)\\
 &=&\int_{\mathbb{R}} \sum_{m=0}^\infty \frac{(2n)^m\, \Delta^m_{k,M^{-1}}f(x)}{m!}\, \overline{f(x)}\,d\mu_k(x)\\
 &\le& \|f\|_{L_k^q(\mathbb{R})}\,\Bigg\| \sum_{m=0}^\infty \frac{(2n)^m\, \Delta^m_{k,M^{-1}}f}{m!} \Bigg\|_{L_k^p(\mathbb{R})}.
\end{eqnarray*}
Taking limit infimum on both sides as $n$ goes to $\infty$, we have
\begin{equation} \label{e:6.4}
e^{-\delta_f} \le  \liminf_{n \to \infty} \Bigg\| \sum_{m=0}^\infty \frac{(2n)^m\, \Delta^m_{k,M^{-1}}f}{m!} \Bigg\|_{L_k^p(\mathbb{R})}^{\frac{1}{2n}}.
\end{equation}
Manipulating \eqref{e6.2} and using H\"older's inequality, we derive the following
\begin{eqnarray*}
 \Bigg \| \sum_{m=0}^\infty \frac{n^m\, \Delta^m_{k,M^{-1}}f}{m!} 
 \Bigg \|^2_{L_k^2(\mathbb{R})} & = & \int_{\mathbb{R}} |e^{-n\lambda^2/b^2}|^2\, |\mathcal{D}_k^M(f)|^2\, d\mu_k(\lambda)
 \end{eqnarray*}
 \begin{eqnarray*}
 &=& \int_{\mathbb{R}} e^{- \lambda^2/b^2} \,\overline{\mathcal{D}_k^M(f)(\lambda)}\,e^{-(2n-1)\lambda^2/b^2}\,\mathcal{D}_k^M(f)(\lambda)\,d\mu_k(\lambda)\\\\
 &=& \int_{\mathbb{R}} \mathcal{D}_k^M\left(\overline{\sum_{m=0}^\infty \frac{\Delta^m_{k,M^{-1}}f}{m!}}\right)(\lambda)\,\mathcal{D}_k^M \left(\sum_{m=0}^\infty \frac{(2n-1)^m\, \Delta^m_{k,M^{-1}}f}{m!}\right) (\lambda)\, d\mu_k(\lambda)\\\\
 &=& \int_{\mathbb{R}}\overline{\sum_{m=0}^\infty \frac{\Delta^m_{k,M^{-1}}f(x)}{m!}}\,\sum_{m=0}^\infty \frac{(2n-1)^m\, \Delta^m_{k,M^{-1}}f(x)}{m!}\, d\mu_k(x)\\\\
 &\le& \Bigg \| \sum_{m=0}^\infty \frac{ \Delta^m_{k,M^{-1}}f}{m!} 
 \Bigg \|_{L_k^q(\mathbb{R})}\,\Bigg \| \sum_{m=0}^\infty \frac{(2n-1)^m\, \Delta^m_{k,M^{-1}}f}{m!} \Bigg \|_{L_k^p(\mathbb{R})}.
 \end{eqnarray*}
 Taking limit infimum on both sides implies that
 \begin{eqnarray}
 \nonumber \liminf_{n \to \infty}\Bigg \| \sum_{m=0}^\infty \frac{n^m\, \Delta^m_{k,M^{-1}}f}{m!} 
 \Bigg \|^{\frac{1}{n}}_{L_k^2(\mathbb{R})} \le&&  \liminf_{n \to \infty} \Bigg \| \sum_{m=0}^\infty \frac{ \Delta^m_{k,M^{-1}}f}{m!} 
 \Bigg \|^{\frac{1}{2n}}_{L_k^q(\mathbb{R})}\\ 
 \nonumber
 \times &&\Bigg \| \sum_{m=0}^\infty \frac{(2n-1)^m\, \Delta^m_{k,M^{-1}}f}{m!} \Bigg \|^{\frac{1}{2n}}_{L_k^p(\mathbb{R})}\\
 \label{e:6.5}
  e^{-\delta_f} \le && \liminf_{n \to \infty} \Bigg \| \sum_{m=0}^\infty \frac{(2n-1)^m\, \Delta^m_{k,M^{-1}}f}{m!} \Bigg \|^{\frac{1}{2n}}_{L_k^p(\mathbb{R})}.  
\end{eqnarray}
Combining \eqref{e:6.4} and \eqref{e:6.5}, we get
\begin{equation} \label{e:6.6}
  e^{-\delta_f} \le \liminf_{n \to \infty} \Bigg \| \sum_{m=0}^\infty \frac{n^m\, \Delta^m_{k,M^{-1}}f}{m!} \Bigg \|^{\frac{1}{2n}}_{L_k^p(\mathbb{R})}.    
\end{equation}
Next, we have to prove the other way inequality
\begin{equation*}
 \limsup_{n \to \infty}\Bigg \| \sum_{m=0}^\infty \frac{n^m\, \Delta^m_{k,M^{-1}}f}{m!} 
 \Bigg \|^{\frac{1}{n}}_{L_k^p(\mathbb{R})} \le e^{-\delta_f}, \quad 1 \le p \le \infty.   
\end{equation*}
For $1\le p \le 2$, applying the same proof argument used in Theorem \ref{t:5.1}, we deduce that
\begin{eqnarray*}
\Bigg \| \sum_{m=0}^\infty \frac{n^m\, \Delta^m_{k,M^{-1}}f}{m!} 
 \Bigg \|_{L_k^p(\mathbb{R})} &\le& C^{\frac{1}{p}}\, \|(I-\Delta^m_{k,M^{-1}}) (e^{-n\lambda^2/b^2}\,\mathcal{D}_k^M(f))\|_{L_k^2(\mathbb{R})}\\
 &\le& \|e^{-n\lambda^2/b^2}\, \Phi_n\|_{L_k^2(\mathbb{R})}
\end{eqnarray*}
where $\text{supp}\,\Phi_n \subset \text{supp}\,\mathcal{D}_k^M(f)$ and $\| \Phi\|_{L_k^2(\mathbb{R})} \le C'\,n^2$.
\begin{eqnarray*}
    \limsup_{n \to \infty} \Bigg \|\sum_{m=0}^\infty \frac{n^m\, \Delta^m_{k,M^{-1}}f}{m!} 
 \Bigg \|_{L_k^p(\mathbb{R})}^{\frac{1}{n}} &\le& \sup_{\lambda \in \text{supp}\,\mathcal{D}_k^M(f)}e^{-(\lambda/b)^2}\, (C'n^2)^{\frac{1}{n}}\\
 &=&e^{-\delta_f}.
\end{eqnarray*}
For $2 \le p \le \infty$, using inverse formula \eqref{e:2.2} and H\"older's inequality, we derive that
\begin{eqnarray*}
\Bigg \|\sum_{m=0}^\infty \frac{n^m\, \Delta^m_{k,M^{-1}}f}{m!} 
 \Bigg \|_{L_k^p(\mathbb{R})}&=& \| \mathcal{D}_k^{M}(e^{-n\lambda^2/b^2})\,\mathcal{D}_k^{M^{-1}}(f)\|_{L_k^p(\mathbb{R})}\\
 &\le&  \frac{1}{\left(|b|^{k+1}\right)^{1-\frac{2}{q}}}\,\|e^{-n\lambda^2/b^2}\,\mathcal{D}_k^{M^{-1}}(f)\|_{L^q_k(\mathbb{R})}\\
 &\le& \frac{1}{\left(|b|^{k+1}\right)^{1-\frac{2}{q}}}\, \sup_{\lambda \in \text{supp}\,\mathcal{D}_k^M(f)}e^{-n(\lambda/b)^2}\,\|\mathcal{D}_k^{M^{-1}}(f)\|_{L^q_k(\mathbb{R})}\\
 \limsup_{n \to \infty}\Bigg \|\sum_{m=0}^\infty \frac{n^m\, \Delta^m_{k,M^{-1}}f}{m!} 
 \Bigg \|_{L_k^p(\mathbb{R})}^{\frac{1}{n}} &\le& \limsup_{n \to \infty}C_{q,b}^{\frac{1}{n}}\,e^{- \delta_f}\, \|\mathcal{D}_k^{M^{-1}}(f)\|_{L^q_k(\mathbb{R})}^{\frac{1}{n}}.
 \end{eqnarray*}
Thus,
\begin{equation*}
     \limsup_{n \to \infty}\Bigg \|\sum_{m=0}^\infty \frac{n^m\, \Delta^m_{k,M^{-1}}f}{m!} 
 \Bigg \|_{L_k^p(\mathbb{R})}^{\frac{1}{n}} \le e^{- \delta_f}, \quad \text{for}\quad 1\le p \le \infty.
\end{equation*}
In particular,
\begin{equation} \label{e:6.7}
  \limsup_{n \to \infty}\Bigg \|\sum_{m=0}^\infty \frac{n^m\, \Delta^m_{k,M^{-1}}f}{m!} 
 \Bigg \|_{L_k^p(\mathbb{R})}^{\frac{1}{2n}} \le e^{- \delta_f}, \quad \text{for}\quad 1\le p \le \infty. 
\end{equation}
Finally, combining \eqref{e:6.6} and \eqref{e:6.7}, we get the desired result.
\end{proof}
\begin{theorem}[Boas-type theorem] \label{t:5.6}
Let $f\in \mathcal{S}(\mathbb{R})$ vanishes in some interval $(-r,r)$ if and only if 
\begin{equation*}
     \lim_{n \to \infty}  \Bigg \| \sum_{m=0}^\infty \frac{n^m\, \Delta^m_{k,M^{-1}}f}{m!} 
 \Bigg \|_{L_k^p(\mathbb{R})}^{\frac{1}{n}} \le e^{-r^2}, \qquad \text{for}\, \,\,\,1\le p\le \infty.
\end{equation*}
\end{theorem}
\begin{proof}
The required inequality directly follows from the second part of the proof of Theorem \ref{t:5.5}.
\end{proof}
\subsection*{Acknowledgments:} 
The second author acknowledges the funding received from Anusandhan National Research Foundation SURE (SUR/2022/005678).
\subsection*{Conflict of interest:} No potential conflict of interest was reported by the author.

\bibliographystyle{amsplain}

\end{document}